\documentclass[12pt, leqno]{amsart}
\usepackage{amsmath}
\usepackage{amssymb}
\usepackage{amsthm}
\usepackage{enumerate}
\usepackage[mathscr]{eucal}
\theoremstyle{plain}
\usepackage{tikz}
\newtheorem{theorem}{Theorem}[section]
\newtheorem{lemma}[theorem]{Lemma}
\newtheorem{prop}[theorem]{Proposition}
\theoremstyle{definition}
\newtheorem{definition}[theorem]{Definition}
\newtheorem{remark}[theorem]{Remark}

\newtheorem{example}[theorem]{Example}

\theoremstyle{remark}

\begin{document}
	\title[Orthogonality induced by norm derivatives ]{Orthogonality induced by norm derivatives : A new geometric constant  and  symmetry}
	\author[Ghosh, Paul  and Sain]{Souvik Ghosh,  Kallol Paul and Debmalya Sain }

	\newcommand{\acr}{\newline\indent}
	
		\address[Ghosh]{Department of Mathematics\\ Jadavpur University\\ Kolkata 700032\\ West Bengal\\ INDIA}
	\email{sghosh0019@gmail.com}
	
	\address[Paul]{Department of Mathematics\\ Jadavpur University\\ Kolkata 700032\\ West Bengal\\ INDIA}
	\email{kalloldada@gmail.com}
	
	\address[Sain]{Department of Mathematics\\ Indian Institute of Information Technology, Raichur\\ Karnataka 584135 \\INDIA}	\email{saindebmalya@gmail.com}

	\thanks{ Souvik Ghosh would like to thank  CSIR, Govt. of India, for the financial support in the form of  Senior Research Fellowship under the mentorship of Prof. Kallol Paul.}


	\begin{abstract}
		In this article we  study the difference between orthogonality induced by the norm derivatives (known as $\rho$-orthogonality) and Birkhoff-James orthogonality in a normed linear space $ \mathbb X$ by introducing a new geometric constant, denoted by $\Gamma(\mathbb{X}).$  We explore the relation between various geometric properties of the space and the constant $\Gamma(\mathbb{X}).$   We also investigate the left symmetric and right symmetric  elements of a normed linear space with respect to $\rho$-orthogonality and obtain a characterization of the same. We characterize inner product spaces among normed linear spaces using the symmetricity of $\rho$-orthogonality. Finally, we provide a complete description of both left symmetric and right symmetric elements with respect to $\rho$-orthogonality for some particular Banach spaces.
	\end{abstract}
	
	\maketitle
	\section{Introduction.} 
	In the study of the geometry of normed linear spaces, geometric constants play a significant role. There are many geometric constants in literature, see \cite{A, GL, kato, K,SGP, ZYWC23} and the references therein. In \cite{Ji}, the authors have developed a geometric constant to study the difference between Birkhoff-James orthogonality and isosceles orthogonality from the quantitative point of view. Later on \cite{papini}, Papini et al.  studied the difference between  Birkhoff-James orthogonality and Robert's orthogonality through another geometric constant. Motivated by these,  we investigate the difference between $\rho$-orthogonality and Birkhoff-James orthogonality by introducing a new constant. Additionally, we also study the symmetric points with respect to $\rho$-orthogonality. Before diving into the main results let us fix the notations and terminologies.

	Letters $\mathbb{X}, \mathbb{Y}$ denote real normed linear spaces and $\mathbb{X}^*$ stands for the dual space of $\mathbb{X}.$ Let $B_\mathbb{X} = \{x \in \mathbb{X}: \|x\|\leq 1\}$ and $S_\mathbb{X}=\{x \in \mathbb{X}: \|x\|=1\}$ denote the unit ball and unit sphere of $\mathbb{X},$ respectively. For a non-empty convex subset $C \subset \mathbb X ,$ an element $x \in C$ is said to be an extreme point of  $C,$ if $x=(1-t)y+tz,$ for some $0<t<1$ and $y, z\in C$ implies $x=y=z.$ The set of all extreme points of $C$ is denoted by $Ext(C).$ A normed linear space $\mathbb{X}$ is said to be strictly convex if $Ext(B_\mathbb{X})=S_\mathbb{X}.$  The collection of all supporting functionals at $x$ is denoted by  $J(x),$ i.e., $J(x)=\{f\in S_{\mathbb{X}^*}: f(x)=\|x\|\}.$ An element $x\in S_\mathbb{X}$ is said to be smooth if $J(x)$ is singleton and the space $\mathbb{X}$ is said to be smooth if  each element of $S_\mathbb{X}$ is smooth. An element $x \in \mathbb{X}$ is said to be Birkhoff-James orthogonal \cite{B, J1} to $y\in \mathbb{X}$,  if $\|x+\lambda y\| \geq \|x\|,$ for all $\lambda \in \mathbb{R}.$ It is denoted as $x \perp_B y.$ From \cite{J} we note that  $x \perp_B y$ if and only if there exists $f \in J(x)$ such that $f(y)=0.$ By $x^{\perp}$ we denote  the collection of all elements which are Birkhoff-James orthogonal to $x,$ i.e., $x^{\perp}=\{y\in \mathbb{X}: x\perp_B y\}.$  Following \cite{sain1}, $x$ is said to be left symmetric with respect to Birkhoff-James orthogonality if for any $y \in \mathbb{X},$ $x \perp_B y \implies y \perp_B x.$ Similarly, $x \in \mathbb{X}$ is said to be right symmetric with respect to Birkhoff-James orthogonality if for any $y \in \mathbb{X},$ $y \perp_B x \implies x \perp_B y.$ A point $x$ is said to be symmetric with respect to Birkhoff-James orthogonality if it is both left and right symmetric with respect to Birkhoff-James orthogonality. Moreover, $\mathbb{X}$ is said to be symmetric with respect to Birkhoff-James orthogonality if $x\perp_B y \implies y \perp_B x$, for all $x, y\in \mathbb{X}.$ For more on Birkhoff-James orthogonality readers may see   the survey article \cite{AMW} and the recent book \cite{MPSbook}.
	
	Let us now mention the  definition of $\rho$-orthogonality studied in \cite{CW,Mi}.  
	
 \begin{definition}
 	Let $\mathbb{X}$ be a normed linear space and let $x, y \in \mathbb{X}.$ The norm derivatives at $x$ in the direction of $y$ is defined as:
 	\begin{eqnarray*}
 		\rho'_{+}(x, y)&=& \|x\| \lim_{t \to 0^+} \frac{\|x+ty\|-\|x\|}{t} \\
 	\rho'_{-}(x, y)&=& \|x\| \lim_{t \to 0^-} \frac{\|x+ty\|-\|x\|}{t}\\
 	\rho'(x, y) &=& \frac{1}{2}(\rho'_+(x, y) + \rho'_-(x, y)).
 	\end{eqnarray*}
 	 We say that $x$ is $\rho$-orthogonal to $y$, i.e., $x \perp_{\rho} y$ if $\rho'(x, y)=0.$ Note that $\rho$-orthogonality is homogeneous, i.e., for any $\alpha, \beta \in \mathbb{R},$ $x \perp_\rho y \iff \alpha x\perp_\rho \beta y.$ For further readings on this topic one can see \cite{AST,CW, CW2, S}.
 \end{definition}
 
 Next we observe some of the important results regarding the functions $\rho'_+$ and $\rho'_-.$ 
 
 \begin{lemma}\label{functional}\cite[Th. 2.4]{Woj}
 	Let $\mathbb{X}$ be a normed linear space. Then for $x, y \in S_\mathbb{X},$ 
 	\begin{eqnarray*}
 		\rho'_+(x, y) &=&\sup \{f(y): f \in Ext(J(x))\},\\
 		\rho'_-(x, y) &=&\inf \{f(y): f \in Ext(J(x))\}.
 	\end{eqnarray*}
 \end{lemma}
 
 \begin{lemma}\cite{Amir}\label{B-J}
 	Let $\mathbb{X}$ be a normed linear space. Then $x \perp_{B} y$ if and only if $\rho'_{-}(x, y) \leq 0 \leq \rho'_{+}(x, y).$ 
 \end{lemma}
Apart from the above mentioned properties of the functions $\rho'_+$ and $\rho'_-,$ interested readers may see \cite{AGT, S}.
  It is a well known fact \cite{CW, CW2} that $\perp_{\rho} \subset \perp_B$ in any normed linear space $\mathbb{X}.$ For the converse inclusion we note the following result. 
 
 \begin{theorem}\label{smooth}\cite[Prop. 2.2.2]{AST}
 	Let $\mathbb{X}$ be a normed linear space. Then $\mathbb{X}$ is smooth if and only if $x \perp_B y$ implies $x \perp_\rho y,$ for all $x, y \in \mathbb{X}.$ 
 \end{theorem}
 
 Observe that if $\mathbb{X}$ is not  a smooth space then $\perp_B$ and $\perp_\rho$ are not equivalent and so it is worth introducing the new constant to study the difference between these two orthogonality, quantitatively.

 \begin{definition}
 	Let $\mathbb{X}$ be a normed linear space. We define the following constant $\Gamma(\mathbb{X})$ as:
 	\[\Gamma (\mathbb{X}) = \sup \big\{|\rho'(x, y)|: x, y \in S_\mathbb{X} \, \mbox{and} \, x \perp_B y\big\}.\]
 \end{definition}
 
We recall the following two well known geometric constants in a normed linear space, which play important roles in this article.
 
 \begin{definition}
 	Let $\mathbb{X}$ be a normed linear space.
 	\begin{enumerate}
 	 \item Then the \emph{James constant} \cite{GL} is defined by 
 	\[
 	J(\mathbb{X})= \sup\big\{\min\{\|x-y\|, \|x+y\|\}: x, y \in S_\mathbb{X}\big\}.
 	\]
 	\item  The \emph{modulus of convexity} is defined as:
 	\[\delta_\mathbb{X}(\epsilon) = \inf\bigg\{1-\frac{\|x+y\|}{2}: x, y\in S_\mathbb{X}, \|x-y\|\geq \epsilon\bigg\},\] where $\epsilon \in [0, 2].$
 	\end{enumerate}
 \end{definition}
 The James constant studies the `uniform non-squareness' of the unit sphere of a normed linear space whereas the modulus of convexity studies the uniform convexity. The space  $\mathbb{X}$ is uniformly non-square if and only if $J(\mathbb{X})<2$ and $\mathbb{X}$ is uniformly convex if and only if  $\delta_\mathbb{X}(\epsilon)>0,$ whenever $\epsilon >0.$ Given any $x, y \in \mathbb{X},$ let us denote the ray passing through $y$ starting from $x$ as $[x, y\rangle,$ which is defined by $[x, y\rangle = \{(1-t)x+ty: t \geq 0\}.$ Following  \cite{BFS}, we mention the positive orientation of a two-dimensional Banach space $\mathbb{X}.$ Suppose that $x=(x_1, x_2), y=(y_1, y_2) \in \mathbb{X},$ where $\mathbb{X}$ is identified with $\mathbb{R}^2$ in the canonical way. Then we say `$x$ precedes $y$,' i.e.,  $x \prec y$ if $x_1y_2-x_2y_1>0.$ In this connection, we would like to mention a very important lemma.
 
 \begin{lemma} (Monotonicity lemma) \label{lem; mon} \cite{MSW}
	Let $\mathbb{X}$ be a two-dimensional Banach space and let $x, y, z \in \mathbb{X}\setminus \{0\}$ such that $x \neq z.$ Suppose that the ray $[0, y\rangle$ lies in between the rays $[0, x\rangle$ and $[0, z\rangle$ with $\|y\|=\|z\|.$  Then $\|x-y\| \leq \|x-z\|.$
	
	Moreover, the inequality is strict if $\mathbb{X}$ is strictly convex.
 \end{lemma}
 
Henceforth, the results of this article are mainly divided into two sections excluding the introductory part. In the first section we explore the newly defined constant $\Gamma(\mathbb{X}).$ We obtain a relation between uniform non-squareness and the constant $\Gamma(\mathbb{X}).$ Then we show that in case of finite-dimensional Banach spaces extreme points are sufficient to estimate $\Gamma(\mathbb{X}).$ Thereafter we give a complete description of $\Gamma(\mathbb{X})$ in case of two-dimensional polygonal Banach space whose unit sphere is a regular 2n-gon. Also, we obtain a necessary condition for uniformly convex Banach space in terms of $\Gamma(\mathbb{X}).$ In the last section we deal with the symmetricity with respect to $\rho$-orthogonality. There we observe the interconnection between $\rho$-symmetricity and symmetricity with respect to Birkhoff-James orthogonality. Further, we obtain a characterization of both $\rho$-left and $\rho$-right symmetric points.  Finally we give a complete description of both the $\rho$-left and $\rho$-right symmetric points of the spaces $\ell_1^n$ and $\ell_\infty^n.$

\section{$\Gamma(\mathbb{X})$ and it's properties.}

In the beginning, we develop a bound for the  constant $\Gamma(\mathbb{X})$. To do so we use the notion of $\mathcal{E}(\mathbb{X}),$ introduced in \cite{CKS}.

 \begin{definition}
 	Suppose that $d: \mathbb{X}\setminus \{0\} \rightarrow \mathbb{R}$ is defined as $d(x) = diam(J(x)),$ where $J(x)$ is the collection of all the supporting linear functionals at $x.$  Then $\mathcal{E}(\mathbb{X})$ is defined as 
 	$$\mathcal{E}(\mathbb{X}) = \sup\{d(x): x\in S_\mathbb{X}\}.$$
 \end{definition}

 \begin{prop}\label{bound}
 	For a normed linear space $\mathbb{X},$ $ 0 \leq \Gamma(\mathbb{X}) \leq  \min\{ \mathcal{E}(\mathbb{X}), \frac{1}{2}\}.$
 \end{prop}

\begin{proof}
	It is easy to see that $\Gamma(\mathbb{X}) \geq 0.$ To obtain the upper bound, we first note from \cite[Th. 2.1.1]{AST} that $|\rho'_{\pm}(x, y)| \leq \|x\|\|y\|.$ Thus $|\rho'(x, y)| \leq 1,$ for any $x, y \in S_\mathbb{X}.$  From Lemma \ref{B-J} we see that when $x \perp_B y,$ we have  $\rho'_{-}(x, y) \leq 0 \leq \rho'_{+}(x, y).$ This implies that $|\rho'(x, y)| \leq \frac{1}{2}. $ Now we show $\Gamma(\mathbb{X}) \leq \mathcal{E}(\mathbb{X}).$ Let us consider any two arbitrary elements $x, y\in S_\mathbb{X}$ such that $x \perp_B y.$ From Lemma \ref{functional}, we note that $\rho'_+(x, y) = \sup \{f(y): f \in Ext(J(x))\}.$ Since $J(x)$ is weak*-compact and convex subset of $\mathbb{X}^*,$  it follows that there exists $f_0 \in J(x)$ such that $\rho'_+(x, y) = f_0(y).$ Similarly, we can obtain that $\rho'_-(x, y) = g_0(y),$ for some $g_0 \in J(x).$ Also, from Lemma \ref{B-J} we note that $f_0(y) \geq 0 \geq g_0(y)$ as $x \perp_B y.$ Thus we have 
	\begin{eqnarray*}
	  |\rho'(x, y)| &=& \frac{1}{2}|\rho'_+(x, y) + \rho'_-(x, y)|\\
	  &=& \frac{1}{2}|f_0(y)+g_0(y)|\\
	  &\leq& \frac{1}{2} |f_0(y) - g_0(y)|\\
	  &\leq& \|f_0 - g_0\| \leq d(x).
	\end{eqnarray*}
	Therefore,
	 \[\Gamma(\mathbb{X}) = \sup\{|\rho'(x, y)|: x, y\in S_\mathbb{X}, x \perp_B y\} \leq \sup\{d(x): x \in S_\mathbb{X}\} = \mathcal{E}(\mathbb{X}).\]
	This completes the proof.
\end{proof}

For any smooth normed linear space we note that $\Gamma(\mathbb{X})=0.$ On the other hand, it is easy to see that $\Gamma(\mathbb{X})=\frac{1}{2},$ when $\mathbb{X} = \ell_\infty^n.$ In fact, taking  $x = (1, 1, \ldots, 1) $ and $y=(0, 0, \ldots, 1),$
we get $\rho'(x, y) = \frac{1}{2}.$  Similarly we can show that $\Gamma(\mathbb{X})=\frac{1}{2},$ when $\mathbb{X} = \ell_1^n.$ Also, we give example of an infinite-dimensional Banach space where $\Gamma(\mathbb{X})=\frac{1}{2}.$

\begin{example}	
	Let us consider the space $c_0$ and let $x=(1, 1, 0, \ldots, 0, \ldots) \in c_0.$ Clearly, $f_1, f_2 \in J(x),$ where for each $i \in \{1, 2\},$ $f_i \in c_0^*$ and $f_i(y) = y_i,$ for all $y =(y_1, y_2, \ldots) \in c_0.$ Take $z=(0, 1, 0, \ldots) \in c_0.$ Clearly, $x \perp_B z.$ Also, we have $f_1(z) =0$ and $f_2(z) = 1.$ Since $x \perp_B z,$ from Lemma \ref{B-J} we have $\rho'_-(x, z) \leq 0 \leq \rho'_+(x, z).$ Also, applying Lemma \ref{functional}, it is easy to observe that $\rho'(x, z) = \frac{1}{2}(\rho'_+(x, z) + \rho'_-(x, z)) = \frac{1}{2}.$ Now from Proposition \ref{bound} one can see that $\Gamma(\mathbb X)=\frac{1}{2}.$ 
	 \end{example}

 Next we prove the following theorem which will be useful to estimate the constant $\Gamma(\mathbb{X})$ in any finite-dimensional polyhedral Banach space.
\begin{theorem}\label{extreme}
	Let $\mathbb{X}$ be an $n$-dimensional Banach space. Then there exists an element $z \in Ext(B_\mathbb{X})$ such that $\Gamma(\mathbb{X})= \rho'(z, w),$ for some $w \in S_\mathbb{X} $ with $z \perp_B w.$
\end{theorem}

\begin{proof}
	Using Carath\'eodory's theorem \cite{R}, we note that for any $x \in S_\mathbb{X},$ there exist $z_1, z_2, \ldots, z_{n+1} \in Ext(B_\mathbb{X})$ such that $x=\sum_{k=1}^{n+1} \lambda_k z_k,$ where $\sum_{k=1}^{n+1} \lambda_k =1$ and  $\lambda_k \geq 0,$ for each $1 \leq k \leq n+1.$ Suppose that $y \in S_\mathbb{X}$ such that  $x\perp_B y.$ Then it is straightforward to see that $z_k \perp_B y,$ for each $1 \leq k \leq n+1.$ Now 
	\begin{eqnarray*}
		2\rho'(x, y) &=& \rho'_+(x, y) + \rho'_-(x, y)\\
		&=& \rho'_+ \Big(\sum_{k=1}^{n+1} \lambda_k z_k, y\Big) + \rho'_-\Big(\sum_{k=1}^{n+1} \lambda_k z_k, y\Big) \\
		&=& \lim_{t \to 0^+} \frac{\|\sum_{k=1}^{n+1} \lambda_k z_k +ty\| -1}{t} + \lim_{t \to 0^-}\frac{\|\sum_{k=1}^{n+1} \lambda_k z_k +ty\| -1}{t}\\
		&=& \lim_{t\to 0^+}  \frac{\|\sum_{k=1}^{n+1} \lambda_k z_k + \sum_{k=1}^{n+1} \lambda_k ty\| -1}{t} + \lim_{t \to 0^-}\frac{\|\sum_{k=1}^{n+1} \lambda_k z_k + \sum_{k=1}^{n+1} \lambda_k ty\| -1}{t}\\
		&\leq& \lim_{t \to 0^+} \frac{\sum_{k=1}^{n+1}\lambda_k \|z_k +ty\| -1}{t} + \lim_{t \to 0^-}\frac{\sum_{k=1}^{n+1} \lambda_k\| z_k +ty\| -1}{t}\\
		&=& \lim_{t \to 0^+} \frac{\sum_{k=1}^{n+1}\lambda_k \|z_k +ty\| - \sum_{k=1}^{n+1} \lambda_k}{t} + \lim_{t \to 0^-}\frac{\sum_{k=1}^{n+1} \lambda_k\| z_k +ty\| - \sum_{k=1}^{n+1} \lambda_k}{t}\\
		&=& \sum_{k=1}^{n+1} \lambda_k (\rho'_{+}(z_k, y) + \rho_{-}(z_k, y))\\
		&\leq& 2\max \{ \rho'(z_k, y)\}.
	\end{eqnarray*}
	This clearly shows that for any $x \in S_\mathbb{X},$ there exists $z\in Ext(B_\mathbb{X})$ such that $\rho'(x, y) \leq \rho'(z, y).$ This completes the proof of the theorem. 
	
\end{proof}

A normed linear space $\mathbb{X}$ is uniformly non-square if $ \sup_{x, y \in S_\mathbb{X}} \min \{\|x-y\|, \|x+y\|\} < 2.$ Note that the spaces $\ell_1^n, \ell_\infty^n$ are non uniformly non-square. Then it is natural to ask  whether for any non uniformly non-square space $\mathbb{X},$  $\Gamma(\mathbb{X})=\frac{1}{2}.$ To proceed in this direction
first we prove the following lemma.  See \cite[Prop. 2.6]{GL}. 

\begin{lemma}\label{lemma}
	Let $\mathbb{X}$ be a two-dimensional Banach space and let $\mathbb{X}$ be non uniformly non-square. Then $\mathbb{X}$ is isometrically isomorphic to $\ell_\infty^2.$
\end{lemma}

\begin{proof}
	Since $\mathbb{X}$ is non uniformly non-square, it follows that there exists $x_0, y_0 \in S_\mathbb{X}$ such that $\min\{\|x_0-y_0\|, \|x_0+y_0\|\} =2,$ i.e., $\|x_0-y_0\| = \|x_0+y_0\| =2.$ Clearly, $x_0 \neq \pm y_0.$ Define a linear map $T : \mathbb{X} \to \ell_\infty^2$ by $Tx_0 =(1, 1)$ and $Ty_0 = (-1, 1).$ Since $\mathbb{X}$ is two-dimensional, for any $z \in \mathbb{X},$  we have $z= \alpha x_0 + \beta y_0,$ where $\alpha, \beta \in \mathbb{R}.$ Then $Tz= T(\alpha x_0 +\beta y_0) = (\alpha-\beta, \alpha + \beta).$ Note that $\|(\alpha-\beta, \alpha+\beta)\|_{\infty} = |\alpha|+|\beta|.$ Thus we only need to show that $\|\alpha x_0 + \beta y_0\| = |\alpha| + |\beta|,$ for any $\alpha, \beta \in \mathbb{R}.$ Since $\frac{1}{2}\|x_0-y_0\|=\frac{1}{2}\|x_0+y_0\|=1,$  $L[x_0, y_0]:=\{(1-t)x_0 + t y_0 : 0 \leq t \leq 1\}$ and $L[x_0, -y_0]$ both are subsets of $S_\mathbb{X}.$   Note that if $\alpha=0$ or $\beta = 0,$ then we are done. Let $\alpha, \beta \neq 0. $ Moreover, assume that $\alpha, \beta > 0.$ Let $z_0 = \frac{\alpha x_0 + \beta y_0}{\|\alpha x_0 + \beta y_0\|}.$ Clearly, $z_0\in S_\mathbb{X}.$ Consider the element $z'=\frac{\|\alpha x_0 + \beta y_0\|}{\alpha + \beta}z_0.$ It is easy to see that $z' \in L[x_0, y_0].$ Since $L[x_0, y_0] \subset S_\mathbb{X},$ it follows that $\frac{\|\alpha x_0 + \beta y_0\|}{\alpha + \beta}=1,$ i.e., $\|\alpha x_0 + \beta y_0\|=\alpha + \beta= |\alpha|+|\beta|.$ Let us now consider $\alpha >0$ and $\beta < 0.$ Then we write $z= \alpha x_0 +\beta y_0 = \alpha x_0 + \beta' (-y_0),$ where $\beta'= -\beta.$ Then we get $\alpha, \beta' > 0.$ Proceeding similarly as above we obtain $\|\alpha x_0 + \beta y_0\|=\|\alpha x_0 + \beta' (-y_0)\|=\alpha + \beta'= |\alpha|+|\beta|.$ Also, the other cases for $\alpha$ and $\beta$ follows similarly as above. This completes the proof. 
	
\end{proof}

It is well known that in a normed linear space the James constant, $J(\mathbb{X})$ studies non uniformly non-squareness of the unit sphere. In the next theorem we obtain a connection between the notion of uniform non-squareness and the constant $\Gamma(\mathbb{X}).$

\begin{theorem}\label{thm1}
	Let $\mathbb{X}$ be a finite-dimensional Banach space. Then $\mathbb{X}$ is uniformly non-square whenever $\Gamma(\mathbb{X}) < \frac{1}{2}.$ 
\end{theorem}

\begin{proof}
	
	Suppose on the contrary that $\mathbb{X}$ is not uniformly non-square. Then from \cite[Th. 3.4]{GL} we note that $J(\mathbb{X})= 2,$ i.e., $\sup \{\min \{\|x+y\|, \|x-y\|\}: x, y\in S_\mathbb{X}\}=2.$ Since $\mathbb{X}$ is finite-dimensional, it follows that there exist $x_0, y_0\in S_\mathbb{X}$ such that $\min \{\|x_0+y_0\|, \|x_0-y_0\|\}=2.$   Clearly, $x_0 \neq \pm y_0.$ Consider the two-dimensional subspace $\mathbb{Y}=span\{x_0, y_0\}.$ Then from Lemma \ref{lemma} it follows that $\mathbb{Y}$ is isometrically isomorphic to $\ell_\infty^2.$ As  $\Gamma(\ell_\infty^2)=\frac{1}{2},$ we get $\Gamma(\mathbb{X}) \geq \Gamma(\mathbb{Y})=\Gamma(\ell_\infty^2)=\frac{1}{2}.$ Thus following Proposition \ref{bound}, we obtain that $\Gamma(\mathbb{X})=\frac{1}{2}.$ This completes the proof of the theorem.
	
\end{proof}

	Let $\mathbb{X}$ be a normed linear space such that $M_{J} \neq \emptyset,$ where 	\[M_{J}:=\{(x, y)\in S_\mathbb{X} \times S_\mathbb{X}: \min\{\|x-y\|, \|x+y\|\} = J(\mathbb{X})\}.\] Then using same arguments as in the proof of Theorem \ref{thm1} we can show that $\mathbb{X}$ is uniformly non-square if $\Gamma(\mathbb{X})< \frac{1}{2}.$

\begin{remark}
	
	(i)   We give an example to show that the result is not true for infinite dimensional space. From   \cite[Th. 1.1]{James}  it follows that a if the unit ball of a normed linear space  is uniformly non-square then the space  is reflexive.  Consider the non-reflexive smooth Banach space $\mathbb{X}$ as given in \cite[Ex. 5.4.13]{M}. 	Then  $\Gamma(\mathbb{X})=0$ (being smooth) but $\mathbb{X}$ is not uniformly non-square (being non-reflexive). \\
	(ii) Next we give an example to show that the converse of Theorem \ref{thm1} is not true, in general. Let us consider the two-dimensional Euclidean space $\mathbb{R}^2,$ endowed with the norm $\ell_1-\ell_\infty.$ Let  $x=(1, 0) \in Ext(B_\mathbb{X})$ and $y=(0, 1).$ Clearly, $x \perp_B y.$ It is easy to calculate that $\rho'(x, y) = \frac{1}{2}.$ Following Proposition \ref{bound} together with Theorem \ref{extreme} we conclude that $\Gamma(\mathbb{X})= \frac{1}{2},$ though $(\mathbb{R}^2, \|\cdot\|_{\ell_1-\ell_\infty})$ is uniformly non-square.

\end{remark}

Applying Theorem \ref{extreme} we next  compute the constant $\Gamma(\mathbb{X})$ for a two-dimensional Banach space whose unit sphere is a regular $2n$-gon. Let us first observe the following proposition.

\begin{prop}\label{decreasing}
	Let $\mathbb{X}$ be a two-dimensional Banach space and let $x \in S_\mathbb{X}.$ Suppose that $w_1, w_2\in S_\mathbb{X}$ satisfying $x \prec w_1 \prec w_2 \prec -x.$ Then $\rho'(x, w_1) \geq \rho'(x, w_2).$
\end{prop}

\begin{proof}
	Note that the ray $[0, w_1\rangle$ lies in between the rays $[0, x\rangle$ and  $[0, w_2\rangle.$ Then applying Lemma \ref{lem; mon} we obtain the following:
	\begin{itemize}
		\item[(1)] $\|x+ tw_1\| \geq \|x+tw_2\|,$ when $t > 0.$
		\item[(2)] $\|x+tw_1\| \leq \|x+tw_2\|,$ when $t<0.$
	\end{itemize}
	Therefore, it is easy to observe from (1) and (2) that 
	\[\lim_{t \to 0^{\pm}} \frac{\|x+tw_1\|-1}{t} \geq \lim_{t \to 0^{\pm}} \frac{\|x+tw_2\|-1}{t}.\] This implies that $\rho'_{\pm}(x, w_1) \geq \rho'_{\pm}(x, w_2)$ and consequently we conclude that $\rho'(x, w_1) \geq \rho'(x, w_2).$ 
	
\end{proof}	

We next note that in a two-dimensional Banach space $ \mathbb{X},$ $ x^{\perp} $ can be described in terms of a normal cone $K.$ To be precise, $ x^{\perp} = K \cup (-K).$ Let us recall that a subset $K$ of $\mathbb{X}$ is said to be a normal cone in $\mathbb{X}$ if it satisfies the following: 
\begin{itemize}
	\item[(i)] $K+K \subset K$
	 \item[(ii)] $\alpha K \subset K,$ for all $\alpha \geq 0$ and
	 \item[(iii)] $K \cap (-K) = \{0\}.$
	\end{itemize}
	  We say that the cone $K$ is determined by $x_1, x_2 \in S_\mathbb{X}$ if  $K \cap S_\mathbb{X}=\bigg\{\frac{(1-t)x_1+tx_2}{\|(1-t)x_1+tx_2\|}: 0 \leq t \leq 1\bigg\}.$ In particular, $K=\{\alpha x_1 +\beta x_2: \alpha, \beta \geq 0\},$  see \cite{SPM}.  	In the following lemma we explicitly find the points  which determine the cone corresponding to the orthogonal region of a vertex for a regular $2n$-gon. 

\begin{lemma}\label{lem3}
	Let $\mathbb{X}$ be a two-dimensional Banach space whose unit sphere is a regular $2n$-gon. Let $\{v_1, v_2, \ldots, v_{2n}\}$ be the vertices of $B_\mathbb{X}$ such that for each $1 \leq j \leq 2n,$ $v_j=(\cos\frac{j-1}{n}\pi, \sin\frac{j-1}{n}\pi).$  Suppose that for some $m \in \{1, 2, \ldots, 2n\},$ $v_m^{\perp} = K \cup (-K),$ where $K$ is a normal cone determined by $w_1$ and $ w_2.$ Then the following holds true:
	\begin{itemize}
		\item[(i)] $w_1= v_{\frac{n+2m-1}{2}}$ and $w_2 = v_{\frac{n+2m+1}{2}},$ when $n$ is odd.
		\item[(ii)] $w_1=\frac{1}{2}(v_{\frac{n+2m-2}{2}}+v_{\frac{n+2m}{2}})$ and $w_2= \frac{1}{2}(v_{\frac{n+2m}{2}}+v_{\frac{n+2m+2}{2}}),$ when $n$ is even.
	\end{itemize} 
\end{lemma}

\begin{proof}
	By a straightforward computation, one can observe that given any $i \in \{1, 2, \ldots, 2n\}$ and for any $u \in [v_i, v_{i+1}],$ the supporting functional of $u$ is given by 
	\begin{equation}
		f(x, y)= \frac{1}{\cos \frac{\pi}{2n}} \bigg(x \cos\frac{2i-1}{2n} \pi+ y \sin\frac{2i-1}{2n} \pi\bigg),
	\end{equation}
	for any $(x, y) \in \mathbb{X}.$ Thus $Ext(J(v_m))=\{f_1, f_2\},$ where 	
	\[f_1(x, y)= \frac{1}{\cos \frac{\pi}{2n}} \bigg(x \cos\frac{2m-3}{2n}\pi+ y \sin\frac{2m-3}{2n}\pi\bigg) \] and \[f_2(x, y)=\frac{1}{\cos \frac{\pi}{2n}}\bigg(x \cos\frac{2m-1}{2n} \pi+ y \sin\frac{2m-1}{2n}\pi\bigg),\] for any $(x, y)\in \mathbb{R}.$  Suppose that $\ker f_i \cap S_\mathbb{X} = \{\pm w_i\},$ for each $i\in \{1, 2\}.$ Observe that $v_m^{\perp} = K \cup (-K),$ where the normal cone $K$ is determined by $\{w_1, w_2\}.$ 	We only find $w_1$ as $w_2$ can be obtained analogously. 	Let $w_1= (1- \lambda)v_j + \lambda v_{j+1} =  (1-\lambda) \big(\cos\frac{j-1}{n}\pi, \sin \frac{j-1}{n}\pi\big) + \lambda\big(\cos\frac{j\pi}{n}, \sin \frac{j\pi}{n}\big),$ for some $\lambda \in [0, 1]$ and for some $j \in \{1, 2, \ldots, 2n\}.$ Since $f_1(w_1)=0,$ it follows by a simple computation that the following equation holds true:
	\begin{equation}
		(1-\lambda) \cos\frac{2j-2m+1}{2n} \pi+\lambda\cos\frac{2j-2m+3}{2n}\pi =0.
	\end{equation}
	From the above equation, it follows that  $\frac{2j-2m+1}{2n}\pi \leq \frac{2t+1}{2}\pi \leq \frac{2j-2m+3}{2n}\pi,$ for some $t \in \mathbb{N}\cup \{0\}.$ This implies that $\frac{(2t+1)n-3}{2}+m \leq j \leq \frac{(2t+1)n-1}{2}+m.$ Since $1 \leq j \leq 2n$ and $n\geq 2,$ it is easy to see that $t \in \{0, 1\}.$ Suppose that $t=0.$  Now we note the following two cases:\\
	\emph{Case-I}: Suppose that $n$ is odd. Then clearly, either $j=\frac{n-3}{2}+m$ or $j=\frac{n-1}{2}+m.$ Then putting these values in equation (2) we obtain  $\lambda=1$ or $\lambda =0,$ respectively. In both these cases we obtain $w_1=v_{\frac{n-1}{2}+m}.$ Thus we get $w_1 = v_{\frac{n+2m-1}{2}}.$  Similarly, we  get  $w_2 = v_{\frac{n+2m+1}{2}}.$ \\
	\emph{Case-II}: Suppose that $n$ is even. Then one can observe that $j = \frac{n-2}{2}+m.$ From equation (2) it follows that $\lambda=\frac{1}{2}.$ Then  $w_1= \frac{1}{2}\big(v_{\frac{n+2m-2}{2}}+v_{\frac{n+2m}{2}}\big).$ Proceeding as before we get $w_2= \frac{1}{2}\big(v_{\frac{n+2m}{2}}+v_{\frac{n+2m+2}{2}}\big).$ This proves (ii).
	
	If $t=1,$ then it is easy to see that $w_1=-v_{\frac{n+2m-1}{2}}$ and $w_2= -v_{\frac{n+2m+1}{2}},$ when $n$ is odd. On the other hand, when $n$ is even, we get $w_1=-\frac{1}{2}\big(v_{\frac{n+2m-2}{2}}+v_{\frac{n+2m}{2}}\big)$ and $w_2= -\frac{1}{2}\big(v_{\frac{n+2m}{2}}+v_{\frac{n+2m+2}{2}}\big).$ This shows that $K \cup (-K)$ is completely determined by $w_1, w_2$ as given in (i) and (ii). This completes the proof of the lemma.
	\end{proof}

In the following theorem we compute the value of $\Gamma(\mathbb{X})$ whenever $\mathbb{X}$ is a two-dimensional Banach space whose unit sphere is a regular $2n$-gon.

\begin{theorem}\label{th}
	Let $\mathbb{X}$ be a two-dimensional Banach space and let $S_\mathbb{X}$ be a regular $2n$-gon, where $n \geq 2.$ Then the following results hold true:
	\begin{itemize}
		
		\item[(i)] $\Gamma(\mathbb{X})= \frac{\cos\frac{n-2}{2n}\pi}{2\cos\frac{\pi}{2n}},$ when $n$ is odd.
		\item[(iI)] $\Gamma(\mathbb{X})=\frac{1}{4\cos\frac{\pi}{2n}}\bigg(\cos\frac{n-3}{2n}\pi+ \cos\frac{n-1}{2n}\pi\bigg),$ when $n$ is even.
		
	\end{itemize}
\end{theorem}

\begin{proof}
	Suppose $\mathbb{X}$ is such that $S_\mathbb{X}$ is a regular $2n$-gon with the vertices $v_1, v_2, \ldots, v_{2n},$ where  $v_j =\bigg( \cos\frac{j-1}{n}\pi, \sin\frac{j-1}{n}\pi\bigg),$ for each $j \in \{1, 2, \ldots, 2n\}.$ Moreover, from Theorem \ref{extreme} there exists an element $z \in Ext(B_\mathbb{X})$ such that $\rho'(z, y)= \Gamma(\mathbb{X}),$ for some $y \in S_\mathbb{X} $ with $z\perp_B y.$ Note that in this case for any $k \in \mathbb{N},$  $\frac{k\pi}{n}$- rotation is an isometric isomorphism on $\mathbb{X}.$ Since Birkhoff-James orthogonality is preserved under isometric isomorphism \cite{Koldobsky}, we only find $\rho'(z, y)$ for a fixed vertex $z,$ where $y \in z^{\perp}$. Without loss of generality we may indeed assume that  $z=v_1 =(1, 0).$ 
	Suppose that $v_1^{\perp}=K \cup (-K),$ where $K$ is a normal cone determined by $y_1, y_2.$ Let us take $y\in K \cap S_\mathbb{X}.$ Note that $v_1 \prec y_1 \preceq y \preceq  y_2 \prec -v_1.$ From Proposition \ref{decreasing}, $\rho'(v_1, y_1) \geq \rho'(v_1, y) \geq \rho'(v_1, y_2).$ Therefore, $\Gamma(\mathbb{X}) = \max \{|\rho'(v_1, y_1)|, |\rho'(v_1, y_2)|\}.$ From the definition it is easy to verify that $\rho'_{+}(v_1, y_2) = \rho'_{-}(v_1, y_1)=0.$ Thus we only find the values of $\rho'_{+}(v_1, y_1)$ and $\rho'_{-}(v_1, y_2).$ We consider the following two cases:

	\textbf{Case I:} Suppose that $n$ is odd. From Lemma \ref{lem3} we see that $y_1= v_{\frac{n+1}{2}}$  and $y_2=v_{\frac{n+3}{2}}.$ Let $Ext(J(v_1))=\{f_1, f_2\},$ where $\ker f_i=\{\pm y_i\},$ for each $1 \leq i \leq 2.$ From Equation (1) we observe that $f_1(x, y)= x - y \tan\frac{\pi}{2n}$ and $f_2(x, y)=x+ y\tan\frac{\pi}{2n},$ for all $x, y\in \mathbb{R}.$ Therefore,  by Lemma \ref{functional} we have
\[
\rho'_+(v_1, y_1)=f_2(y_1)= \cos\frac{n-1}{2n}\pi + \tan\frac{\pi}{2n} \sin\frac{n-1}{2n}\pi.	\]
By simplifying, the above equation reduces to 
\[\rho'_+(v_1, y_1)=f_2(y_1) = \frac{\cos\frac{n-2}{2n}\pi}{\cos\frac{\pi}{2n}} .\]
Also, \[\rho'_-(v_1, y_2)=f_1(y_2) = -\frac{\cos\frac{n-2}{2n}\pi}{\cos\frac{\pi}{2n}} .\] Considering these together we get:
\[
|\rho'(v_1, y_1)|=|\rho'(v_1, y_2)|= \frac{\cos\frac{n-2}{2n}\pi}{2\cos\frac{\pi}{2n}}.
\]
This proves (i).\\

   \textbf{ Case-II:} Suppose that $n$ is even. Then from Lemma \ref{lem3} we get $y_1= \frac{1}{2}(v_{\frac{n}{2}}+v_{\frac{n+2}{2}})$ and $y_2=\frac{1}{2}(v_{\frac{n+2}{2}}+v_{\frac{n+4}{2}}).$ Let $Ext(J(v_1))=\{f_1, f_2\},$ where $f_1, f_2$ are same as in Case-I. Then 
	\[\rho'_+(v_1, y_1)=f_1(y_2) = \frac{1}{2\cos\frac{\pi}{2n}}\bigg(\cos\frac{n-3}{2n}\pi + \cos\frac{n-1}{2n}\pi\bigg).\]
	Proceeding similarly we obtain that
	\[
	\rho'_{-}(v_1, y_2) = f_2(y_1) = -\frac{1}{2\cos\frac{\pi}{2n}}\bigg(\cos\frac{n-3}{2n}\pi + \cos\frac{n-1}{2n}\pi\bigg).
	\] Thus we see that 
	\[\Gamma(\mathbb{X})= \max\{|\rho'(v_1, y_1)|, | \rho'(v_1, y_2)|\} = |\rho'(v_1, y_1)|= \frac{1}{2}\rho'_{+}(v_1, y_1). \]
	This proves (ii).

	 Hence the proof of the theorem is completed.
\end{proof}

 Let us now calculate the value of $\Gamma(\mathbb{X}),$ for some particular two-dimensional Banach spaces.

\begin{example}
	\begin{itemize}
\item[(i)] Let $\mathbb{X}$ be a Banach space such that $S_\mathbb{X}$ is a regular octagon. Then we have $n=4.$ Applying Theorem \ref{th}(i) we have $\Gamma(\mathbb{X})=\frac{1}{2\sqrt{2}}.$

\item[(ii)]	Let $\mathbb{X}$ be a two-dimensional Banach space, endowed with the norm $\ell_p-\ell_1.$ For any $(x, y)\in \mathbb{X},$ 
	\begin{eqnarray*}
		\|(x, y)\| &=& (|x|^p+|y|^p)^{\frac{1}{p}}, \, \mbox{whenever} \, xy\geq 0\\
		&=& (|x| + |y|), \, \mbox{whenever} \, xy\leq 0 
	\end{eqnarray*}
	
	Then $\Gamma(\mathbb{X})=\frac{1}{2},$ where $1 \leq p\leq \infty$.
	It is clear that $e_1=(1, 0), e_2=(0, 1) \in \mathbb{X}.$ Moreover, $\|(1, 0)\|=\|(0, 1)\| = 1.$   Note that $\rho'_+(e_1, e_2)=\lim_{t \to 0^+} \frac{(1+t^p)^{\frac{1}{p}} - 1}{t}.$ Thus we obtain that $\rho'_+(e_1, e_2) = 0.$ On the other hand, $\rho'_-(e_1, e_2) = \lim_{t \to 0^-} \frac{1 +|t| -1}{t}.$ This implies that $\rho'_{-}(e_1, e_2) = -1.$ Therefore, $|\rho'(e_1, e_2)|=\frac{1}{2}.$	Similarly, we can show that $\Gamma(\ell_p^2-\ell_\infty^2)= \frac{1}{2}.$
	\end{itemize}
\end{example}

We end this section with the estimation of the constant $\Gamma(\mathbb{X})$ for uniformly convex Banach spaces.

\begin{theorem}\label{uc}
	Let $\mathbb{X}$ be a uniformly convex Banach space. Then $\Gamma(\mathbb{X}) < \frac{1}{2}.$
\end{theorem}

\begin{proof}
	Suppose on the contrary that $\Gamma(\mathbb{X})=\frac{1}{2}.$ Then there exist two sequences $\{x_n\}_{n \in \mathbb{N}}$, $\{y_n\}_{n \in \mathbb{N}} \subset S_\mathbb{X}$ such that $x_n \perp_B y_n$ and $|\rho'(x_n, y_n)| \to \frac{1}{2}.$ Since $x_n \perp_B y_n,$ it follows from Lemma \ref{B-J} that,  for each $n \in \mathbb{N},$  $-1 \leq \rho'_-(x_n, y_n) \leq 0 \leq\rho'_+(x_n, y_n)\leq 1.$ This implies that either of the following two holds true:
	\begin{itemize}
		\item[(1)] $\rho'_+(x_n, y_n) \to 1$ and $\rho'_-(x_n, y_n) \to 0,$ as $n \to \infty.$
		\item[(2)] $\rho'_+(x_n, y_n) \to 0$ and $\rho'_-(x_n, y_n) \to -1,$ as $n \to \infty.$
	\end{itemize}
	Without loss of generality we assume that (1) holds true. Then from Lemma \ref{functional}, we have $\lim_{n \to \infty}\big\{\sup\{f_{x_n}(y_n) : f_{x_n}\in Ext(J(x_n))\}\big\}=1.$ Then for each $n \in \mathbb{N},$ $\|x_n+y_n\| \geq |f_{x_n}(x_n +y_n)| \geq 1+ f_{x_n}(y_n).$ Thus
	\begin{eqnarray*}
		\|x_n + y_n\| &\geq& \sup\{1+ f_{x_n}(y_n): f_{x_n} \in Ext(J(x_n))\}\\
		&=& 1+ \sup\{f_{x_n}(y_n): f_{x_n} \in Ext(J(x_n))\}.
	\end{eqnarray*}
		Taking limit on the both side of the above inequality, we get that $\lim_{n \to \infty}\|x_n+y_n\| \geq 2.$ Also, we have $\|x_n +y_n\|\leq 2,$ for each $n.$ This implies that $\lim_{n \to \infty}\|x_n+y_n\| = 2.$ On the other hand, since for each $n,$ $x_n \perp_B y_n,$ it follows that $\|x_n -y_n\| \geq 1.$ Therefore, $\|x_n-y_n\| \not\to 0.$ From \cite[Prop. 5.2.8]{M}, this contradicts the fact that $\mathbb{X}$ is uniformly convex. 
	\end{proof}

The converse of Theorem \ref{uc} is not true, in general.  There are spaces for which $ \Gamma (\mathbb X) < \frac{1}{2}$  but the spaces are not
uniformly convex (see Theorem \ref{th}).

\section{Symmetric properties of $\rho$-orthogonal elements:}

Following the notion of left and right symmetric points with respect to Birkhoff-James orthogonality, introduced and studied in \cite{sain1}, we now define $\rho$-left and $\rho$-right symmetric points. Given any $x \in \mathbb{X},$ we say $x$ is $\rho$-left symmetric ($\rho$-right symmetric)  if $x \perp_\rho y$ implies $y \perp_\rho x$ $(y\perp_\rho x \, \, \mbox{implies}\, \, x\perp_\rho y),$ for all $y \in \mathbb{X}.$ If $x$ is  both $\rho$-left and $\rho$-right symmetric then we say that $x$ is $\rho$-symmetric.  The space $\mathbb{X}$ is said to be $\rho$-symmetric if for any $x, y \in \mathbb{X},$ we have $x \perp_\rho y \implies y \perp_\rho x.$ If $dim(\mathbb{X}) \geq 3$ and Birkoff-James orthogonality is symmetric then the norm on $\mathbb{X}$ is induced by an inner product (see \cite{Day, J2}). However, if $dim (\mathbb{X})=2,$ then there exists spaces where the Birkhoff-James orthogonality is symmetric but the norm is not necessarily induced by an inner product. A two-dimensional Banach space where Birkhoff-James orthogonality is symmetric is known as the Radon plane. In this section we focus on the study of $\rho$-symmetric points and $\rho$-symmetric spaces. We begin with the following theorem.


\begin{theorem}\label{two-sym}
	
		Let $\mathbb{X}$ be a two-dimensional Banach space and let $\mathbb{X}$ be $\rho$-symmetric. Then $\mathbb{X}$ is strictly convex.
\end{theorem}

\begin{proof}
	Suppose on the contrary  $\mathbb{X}$ is not strictly convex. Then there exist $u, v\in S_\mathbb{X}$ such that the closed line segment $L[u, v]:=\{(1-t)u+tv: 0 \leq t \leq 1\}$ is a subset of the unit sphere of $\mathbb{X}.$ There exists unique $f \in S_{\mathbb{X}^*}$ such that $f(x) =1,$ for all $x \in L[u, v].$ In other words, $f$ supports the line $L[u, v].$ Consider that $\ker f \cap S_\mathbb{X} = \{\pm y\}.$ Then for any $x \in L[u, v],$  $x \perp_B y.$ We take $x  \in L(u, v),$ where $L(u, v):=\{(1-t)u+tv: 0 < t<1\}.$  Since $x $ is a smooth point, it follows that $x \perp_\rho y.$ Since $\mathbb{X}$ is $\rho$-symmetric, it follows that $y \perp_\rho x.$ Let $Ext(J(y)) = \{g, h\}.$ Then one can observe using Lemma \ref{functional} that $y \perp_\rho w$  if and only if $w \in \ker (g+h).$  Therefore, $L(u, v) \subset \ker (g+h).$ This is a contradiction. Thus $\mathbb{X}$ is strictly convex.
\end{proof}

Using the above theorem we observe the following result.

\begin{theorem}\label{sym}
	Let $\mathbb{X}$ be a normed linear space. 
	\begin{itemize}
	\item[(i)] Suppose that $dim(\mathbb{X})=2.$ If $\mathbb{X}$ is $\rho$-symmetric then $\mathbb{X}$ is a Radon plane.
	 \item[(ii)] Suppose $dim(\mathbb{X}) \geq 3.$ Then $\mathbb{X}$ is $\rho$-symmetric if and only if $\mathbb{X}$ is an inner product space.
	 \end{itemize}
\end{theorem}

\begin{proof}
	(i) We prove that if $\mathbb{X}$ is $\rho$-symmetric then $\mathbb{X}$ is symmetric with respect to Birkhoff-James orthogonality. Suppose on the contrary that there exist $x, y \in S_\mathbb{X}$ such that $x \perp_B y$ but $y \not\perp_B x.$ Then clearly, $y \not\perp_\rho x.$  Let us consider a nonzero real number $\alpha = -\rho'(y, x).$ It is easy to see that $\rho'(y, \alpha y +x) =0.$ Take $z = \frac{\alpha y+x}{\|\alpha y+x\|} \in S_\mathbb{X}.$ Then $y \perp_\rho z.$ Since $\mathbb{X}$ is $\rho$-symmetric then $z \perp_\rho y.$ This implies that $z\perp_B y. $ Therefore, there exists $f \in J(z)$ such that $y \in \ker f.$ Also, $x \perp_B y$ implies that there exists $g \in J(x)$ such that $y \in \ker g. $ Therefore, $y \in \ker f \cap \ker g.$ From Theorem \ref{two-sym} we note that  $\mathbb{X}$ is strictly convex.  Therefore, $J(z) \cap J(x) = \emptyset.$ This shows that $f $ and $g$ are linearly independent. Thus  we obtain that $y  = 0,$ which is a contradiction. This implies that $\mathbb{X}$ is symmetric with respect to Birkhoff-James orthogonality and therefore it must be a Radon plane.\\
	
	(ii) The sufficient part follows trivially. We prove the necessary part. Since $\mathbb{X}$ is $\rho$-symmetric, it follows that every two-dimensional subspace of $\mathbb{X}$ is $\rho$-symmetric. Then applying  Theorem \ref{two-sym}, every two-dimensional subspace of $\mathbb{X}$ is symmetric with respect to Birkhoff-James orthogonality. This implies that $\mathbb{X}$ is symmetric with respect to Birkhoff-James orthogonality. Hence from \cite[Th. 6.4]{Day} it follows that $\mathbb{X}$ is an inner product space.
\end{proof}

In the next example we see that the converse of Theorem \ref{sym}(i) is not true.

\begin{example}
	Let us consider the two-dimensional Radon plane $(\mathbb{R}^2, \|\cdot\|_{\ell_1-\ell_\infty}).$ Observe that all the points on the unit sphere are symmetric with respect to Birkhoff-James orthogonality but  there are many points which are not symmetric with respect to $\rho$-orthogonality. Note that $(1, 0) \in \mathbb{R}^2$ is not a  $\rho$-symmetric point. Indeed, take $(-\frac{1}{3}, 1)\in\mathbb{R}^2.$ Then it is clear to see that $(-\frac{1}{3}, 1) \perp_\rho (1, 0)$ whereas, $(1, 0) \not \perp_\rho (-\frac{1}{3}, 1).$ So, $\rho$-orthogonality is not symmetric.
\end{example}

 While investigating Birkhoff-James orthogonality, the notions of $x^+$ and $x^-$ were elegantly introduced by Sain \cite{sain}. Motivated by these here we introduce the notions of $x^{\rho+}$ and $x^{\rho-}$ as follows:

\begin{definition}
	Let $\mathbb{X}$ be a normed linear space and let $x, y \in \mathbb{X}.$ We say $y \in x^{\rho+}$ if $\rho'(x, y) \geq 0$ and $y \in x^{\rho-}$ if $\rho'(x, y) \leq 0.$
\end{definition}

We state the following proposition for which the proofs are trivial.

\begin{prop}
	Let $\mathbb{X}$ be a normed linear space and let $x, y \in \mathbb{X}.$ Then the following relations hold true:
	\begin{itemize}
	\item [(i)] Either $y \in x^{\rho+} $ or $y \in x^{\rho-}$.
\item[(ii)] $x \perp_\rho y$ if and only if $y \in x^{\rho+}$ and $y \in x^{\rho-}$.
	\item[(iii)] $y \in x^{\rho+}$ implies that $\alpha y \in (\beta x)^{\rho+}$ for all $\alpha, \beta > 0.$
\item	[(iv)] $y \in x^{\rho+} $ implies that $-y \in x^{\rho-}$ and $y \in (-x)^{\rho-}.$
	\item[(v)] $y \in x^{\rho-}$ implies that $\alpha y \in (\beta x)^{\rho-}$ for all $\alpha, \beta > 0.$
	\item[(vi)] $y \in x^{\rho-}$ implies that $-y \in x^{\rho+}$ and $y \in (-x)^{\rho+}.   $
	\end{itemize}
\end{prop}

With the help of the above notions, we obtain the characterization of $\rho$-left symmetric points. 
\begin{theorem}\label{left sym}
	Let $\mathbb{X}$ be any normed linear space and let $x \in S_\mathbb{X}.$ Then $x$ is $\rho$-left symmetric if and only if for any $y \in S_\mathbb{X},$ the following conditions hold true:
	\begin{itemize}
		\item[(i)] $y \in x^{\rho+}$ implies $x \in y^{\rho+}$
		\item[(ii)] $y \in x^{\rho-}$ implies $x \in y^{\rho-}.$
	\end{itemize}
\end{theorem}

\begin{proof}
	Note that the sufficient part is easy. Indeed, let $x \perp_\rho y.$ This implies that $y \in x^{\rho+} \cap x^{\rho-}.$ From the hypothesis we have $x \in y^{\rho+} \cap y^{\rho-}.$ Thus $y \perp_\rho x.$ 
	
	To prove the necessary part we only show condition (i) as condition (ii) can be proved similarly. For this let $y \in x^{\rho+}.$ This implies $\rho'(x, y)\geq 0.$ If $\rho'(x, y) =0$ then we have $\rho'(y, x)=0,$ since $x$ is $\rho$-left symmetric. Thus in this case $x \in y^{\rho+}.$ Next let us assume that  $\rho'(x, y) > 0.$ If $x=y$ then we are done. Also, note that $x\neq -y.$ Thus we assume $x\neq \pm y.$  Let $V = span\{x, y\}$ and let $z = y - \rho'(x, y)x \in V.$ It is easy to observe that $\rho'(x, z) =0,$ i.e., $x \perp_\rho z.$ Since $x$ is $\rho$-left symmetric, it follows that $z \perp_\rho x,$ i.e., $\rho'(z, x) =0.$ Since $y = z + \rho'(x, y)x$ and $\rho'(x, y)> 0,$ it follows that  the ray $[0, y\rangle$ lies in between the rays $[0, x\rangle$ and $[0, z\rangle.$ Let $z' = \frac{z}{\|z\|}.$ Then  the ray $[0, y\rangle$ lies in between the rays $[0, x\rangle$ and $[0, z'\rangle.$ Now applying Lemma \ref{lem; mon} we obtain that for each $t > 0,$
	\begin{eqnarray*}
		\frac{\|z'+tx\|-1}{t} \leq \frac{\|y+tx\| -1}{t}.
	\end{eqnarray*} 
	Taking $t \rightarrow 0^+$  we get $\rho'_+(y, x) \geq \rho'_+(z', x).$ By using similar arguments we can show that $\rho'_-(y, x) \geq \rho'_-(z', x).$ Therefore, we have $\rho'(y, x) \geq \rho'(z', x).$ Since $\rho'(z', x)=0,$ it follows that  $\rho'(y, x) \geq 0.$ Therefore, (i) holds true. Hence the theorem.

	
	
\end{proof}

	We already observed that the characterization of a $\rho$-left symmetric point holds analogously as given in \cite[Th. 2.1]{SRBB}. It is now natural to presume that an analogous version of \cite[Th. 2.2]{SRBB} also holds true in case of $\rho$-right symmetric points. But in the following example we see that some of the $\rho$-right symmetric points behave otherwise.

\begin{example}
  Let us consider the space $\ell_\infty^3$. Suppose that $x = (1, 1, \frac{1}{2}) \in S_{\ell_\infty^3}.$ It is easy to observe that $x$ is a $\rho$-right symmetric point. Let $y=(-\frac{1}{2}, 0, 1) \in \ell_\infty^3.$ Note that $Ext(J(x))= \{f_1, f_2\}$ and $Ext(J(y))=\{f_3\},$ where for each $i \in \{1, 2, 3\},$  $f_i(x)= x_i,$ for all $x=(x_1, x_2, x_3) \in \ell_\infty^3.$ Now applying Lemma \ref{functional} it is easy to obtain that $\rho'(y, x)=\frac{1}{2}>0$ and $\rho'(x, y)= -\frac{1}{4}<0.$ This shows that  $x \in y^{\rho+}$ but $y \in x^{\rho-}.$ 
\end{example}

\begin{remark}
 Given any $x, y \in \mathbb{X},$ we say that $\perp_\rho$ has $\alpha$-left ($\alpha$-right) existence if there exists an $\alpha \in \mathbb{R}$ such that $\alpha x + y \perp_\rho x$ ($x \perp_{\rho} \alpha x+ y$).	Unlike Birkhoff-James orthogonality, $\rho$-orthogonality does not always have the $\alpha$-left existence. From the above example we can observe by a straightforward computation that there does not exist any $\alpha \in \mathbb{R}$ such that $\alpha x+y \perp_\rho x.$  In other words, $\perp_\rho$ does not satisfy the $\alpha$-left existence at $x$.  On the other hand, $\rho$-orthogonality always satisfies the $\alpha$-right existence. 
\end{remark}

Our next aim is to obtain a characterization of $\rho$-right symmetric points for which the $\alpha$-left existence is guaranteed. 

\begin{theorem}
	Suppose that $\mathbb{X}$ is a normed linear space and $x \in \mathbb{X}$  satisfies the $\alpha$-left existence property. Then $x$ is $\rho$-right symmetric if and only if for any $y \in S_\mathbb{X},$ the following conditions hold true:
	\begin{itemize}
			\item[(i)]  $x \in y^{\rho+}$ implies $y \in x^{\rho+}$
		\item[(ii)] $x \in y^{\rho-}$ implies $y \in x^{\rho-}.$
	\end{itemize}
	
\end{theorem}

\begin{proof}
	Since the sufficient part is easy to show, we only prove the necessary part.  We prove Condition (i) as Condition (ii) can be proved similarly. Suppose on the contrary that $x\in y^{\rho+}$ but $y \notin x^{\rho+},$ for some $y \in S_\mathbb{X}.$ This implies that $\rho'(y, x)\geq 0$ but $\rho'(x, y)<0.$ If $\rho'(y, x) =0 $ then by $\rho$-right symmetricity of $x$ we get $\rho'(x, y) = 0.$ In that case we have nothing to prove. So, we consider that $\rho'(y, x) > 0.$	If $x=y$ then we are done. Note  that $x \neq - y.$  Thus we assume $x \neq \pm y.$ Let us consider the two-dimensional subspace $\mathbb{Y}= span\{x, y\}.$ Since $x$ has the $\alpha$-left existence property, it follows that there exists a nonzero $\alpha \in \mathbb{R}$ such that $\rho'(\alpha x+y, x) = 0.$  As $x$ is $\rho$-right symmetric, we have $\rho'(x, \alpha x+ y) = 0.$ This implies that $\rho'(x, y) = -\frac{1}{\alpha}.$ Since $\rho'(x, y)<0,$ we get $\alpha > 0.$  Let us assume $\frac{\alpha x+y}{\|\alpha x+y\|}=w.$ Then $y = \|\alpha x+y\|w-\alpha x.$ Since $\alpha > 0,$ it is easy to see that the ray $[0, w\rangle$ lies in between the rays $[0, x\rangle$ and $[0, y\rangle.$ Now applying  Lemma \ref{lem; mon} and proceeding as in Theorem \ref{left sym}, we obtain the following:
	\begin{itemize}
		\item $\rho'_+(w, x) \geq \rho'_+(y, x)$ and
		\item $\rho'_-(w, x) \geq \rho'_-(y, x).$
	\end{itemize}
	Combining these we get $\rho'(y, x) \leq \rho'(w, x) =0.$ This is a contradiction to the fact that $\rho'(y, x)>0$. This completes the necessary part.
\end{proof}

Next, we give an example of $\rho$-right symmetric points which has $\alpha$-left existence.
\begin{example}
	Suppose that $x \in Ext(B_{\ell_1^n}).$ By an easy computation it  can be observed that $x$ is $\rho$-right symmetric points of $\ell_1^n$ (also, see Theorem \ref{right; sum}).  Now one can check that  given any $y =(y_1, y_2, \ldots, y_n)\in S_{\ell_1^n},$ there exists an $\alpha \in \mathbb{R}$ such that $\alpha x+ y \perp_\rho x.$ Indeed, if $x=(1, 0, \ldots, 0) \in Ext(B_{\ell_1^n})$ then  taking $\alpha=-y_1$ we obtain that $\alpha x+y \perp_\rho x.$
\end{example}  

Now we focus on the study of the $\rho$-symmetric points (both left and right) in the classical $\ell_p^n$ spaces. Note that $\ell_p^n$ is a smooth Banach space whenever $1 < p < \infty.$ Therefore, the Birkhoff-James orthogonality coincides with the $\rho$-orthogonality. Thus the characterization of $\rho$-left and $\rho$-right symmetric points in $\ell_p^n$ follow easily from \cite{CSS}. So we only study the $\rho$-left and $\rho$-right symmetric points in $\ell_1^n$ and $\ell_\infty^n.$ To do so we introduce the notations $\mathcal{Z}_x$ and $\mathcal{I}_x$ for $x=(x_1, x_2, \ldots, x_n) \in \mathbb{R}^n,$ where 
 $\mathcal{Z}_x = \{i \in \{1, 2, \ldots, n\}: x_i=0 \}$  and  $\mathcal{I}_x = \{i \in \{1, 2, \ldots, n\}: |x_i|=1\}.$ Clearly, for any extreme point $x \in \ell_1^n,$ $|\mathcal{Z}_x|=n-1$ and $|\mathcal{I}_x| = 1.$ A point $x \in \ell_1^n$ is smooth if and only if $\mathcal{Z}_x =\emptyset.$ For any extreme point $x \in \ell_\infty^n,$ note that $|\mathcal{I}_x|=n.$ Let us first characterize the $\rho$-orthogonal elements in $\ell_1^n$ and $\ell_\infty^n.$
\begin{prop}\label{ortho; characterization}
	Let $\mathbb{X}=\ell_p^n,$ where $p=1, \infty.$ Let  $x=(x_1, x_2, \ldots, x_n)$ and $y=(y_1, y_2, \ldots, y_n)\in S_\mathbb{X}.$ 
	\begin{itemize}
		\item[(i)] If $p=1,$ then $x \perp_\rho y$ if and only if $\sum_{i=1}^n sgn(x_i)y_i = 0.$ 
		\item[(ii)] If $p=\infty,$ then $x \perp_\rho y$ if and only if $\max_{i \in \mathcal{I}_x}\{sgn(x_i)y_i\} + \min_{i \in \mathcal{I}_x}\{sgn(x_i)y_i\}=0.$ 
	\end{itemize}
\end{prop}

\begin{proof}
	(i) Let $x\perp_\rho y.$ Thus it can be easily observed that 
	\begin{eqnarray*}
	Ext(J(x))=\bigg\{u=(u_1, u_2, \ldots, u_n) \in \ell_\infty^n: u_k&=&sgn(x_k),\, \mbox{when}\, k \notin \mathcal{Z}_x\\
	&\mbox{and}&\\
	u_k&\in& \{\pm 1\},\, \mbox{when} \, k\in \mathcal{Z}_x\bigg\}.
	\end{eqnarray*}
	Since $x\perp_\rho y,$ it follows that $\rho'(x, y)=0,$ i.e., $\rho'_+(x, y) = -\rho'_-(x, y).$ This implies from Lemma \ref{functional} that
	 \begin{eqnarray*}
	 	&&\max\bigg\{\sum_{k=1}^n u_ky_k : u \in Ext(J(x))\bigg\}= -\min\bigg\{\sum_{k=1}^n u_ky_k : u \in Ext(J(x))\bigg\}\\
	 	&\implies& \sum_{k \notin \mathcal{Z}_x} sgn(x_k)y_k + \sum_{k\in \mathcal{Z}_x} |y_k| = -\sum_{k \notin \mathcal{Z}_x} sgn(x_k)y_k + \sum_{k\in \mathcal{Z}_x} |y_k|
	 \end{eqnarray*}
	  Therefore,  $\sum_{k \notin \mathcal{Z}_x} sgn(x_k)y_k=0.$  This means that $ \sum_{k=1}^n sgn(x_k)y_k=0.$
	  The converse part is immediate using similar arguments as above. This completes the proof of (i).\\
	  
	  (ii) Observe that for any $x \in S_{\ell_\infty^n},$ $$Ext(J(x))=\big\{(0, 0, \ldots, sgn(x_i), 0, \ldots, 0): i \in \mathcal{I}_x\big\}.$$ Since $x \perp_\rho y,$ it follows from Lemma \ref{functional} that $\max_{i \in \mathcal{I}_x}\{sgn(x_i)y_i\} + \min_{i \in \mathcal{I}_x}\{sgn(x_i)y_i\}=0.$ This proves (ii).
\end{proof}

 In the following theorem we give a complete description of $\rho$-left symmetric points of $\ell_1^n.$
 
 \begin{theorem}\label{left; sum}
 	Let $x =(x_1, x_2, \ldots, x_n) \in S_{\ell_1^n}.$  Then $x$ is $\rho$-left symmetric if and only if either of the following holds true:
 	\begin{itemize}
 		\item[(i)] $x \in Ext(B_{\ell_1^n}).$
 		\item[(ii)] $|x_i|=|x_j|=\frac{1}{2},$ for some $i, j\in \{1, 2, \ldots, n\}$ and $x_k=0,$ otherwise.
 	\end{itemize}
 \end{theorem}

 \begin{proof}
 	First we prove the sufficient part. Suppose that (i) holds true. Then $x_i=\pm 1,$ for some $i \in \{1, 2, \ldots, n\}$ and $x_j=0,$ for all $j \in \{1, 2, \ldots, n\}\setminus \{i\}.$ Suppose that $x\perp_\rho y,$ for some $y=(y_1, y_2, \ldots, y_n) \in S_{\ell_1^n}.$ Then from Proposition \ref{ortho; characterization} we obtain that $y_i=0.$ Therefore, $\sum_{k=1}^n sgn(y_k)x_k = 0.$ Using Proposition  \ref{ortho; characterization} again, we obtain that $y \perp_\rho x.$ Thus $x$ is $\rho$-left symmetric.
 	Now suppose that (ii) holds true. Also, assume that $x\perp_\rho y,$ for some $y=(y_1, y_2, \ldots, y_n) \in S_{\ell_1^n}.$ Then from Proposition \ref{ortho; characterization} we observe that \begin{eqnarray*}
 	sgn(x_i)y_i + sgn(x_j)y_j = 0  &\implies& sgn(x_i)sgn(y_i)+sgn(x_j)sgn(y_j)=0\\
 	&\implies& sgn(y_i)sgn(x_i)|x_i|+sgn(y_j)sgn(x_j)|x_j|=0\\
 	&\implies& \sum_{k=1}^n sgn(y_k)x_k = 0.
 	\end{eqnarray*}
 	This proves that $y\perp_\rho x.$  Thus the proof of the sufficient part is done.\\
 	
 	Next we prove the necessary part. Suppose on the contrary that $x =(x_1, x_2, \ldots, x_n)\in S_{\ell_1^n}$ does not satisfy (i) and (ii). Then clearly, $|\mathcal{Z}_x^c| \geq 2. $ Suppose that there exist $i, j \in \mathcal{Z}_x^c$ such that $|x_i| \neq |x_j|.$ Let $y=(y_1, y_2, \ldots, y_n) \in \ell_1^n$ be such that $|y_i|=|y_j|$ with $sgn(y_i)= sgn(x_i)$ and $sgn(y_j)= -sgn(x_j)$ and $y_k =0,$ for all $k \in \{1, 2, \ldots, n\} \setminus \{i, j\}.$ Then one can observe from Proposition \ref{ortho; characterization} that $x \perp_\rho y$ whereas, $y \not\perp_\rho x.$ Thus $x$ is not $\rho$-left symmetric. Now suppose that for all $j, k \in \mathcal{Z}_x^c,$  $|x_j|=|x_k|.$ It is trivial to see that $|\mathcal{Z}_x^c| > 2,$ otherwise (i) or (ii) will be satisfied. Without loss of generality assume that $x_j > 0,$ for all $j \in \mathcal{Z}_x^c.$ Suppose that $|\mathcal{Z}_x^c|=r.$  Take $y=(y_1, y_2, \ldots, y_n)\in \ell_1^n$ such that  $y_{k_0} = 1-r,$ for some $k_0 \in \mathcal{Z}_x^c$ and $y_j=1,$ for all $j \in \{1, 2, \ldots, n\}\setminus \{k_0\}.$ Note that $\sum_{j=1}^n sgn(x_j)y_j = \sum_{i \in \mathcal{Z}_x^c} y_j = 0.$ Thus $x \perp_\rho y.$ On the other hand, we can see that $sgn(y_{k_0})=-1$ and $sgn(y_j)=+1,$ for all $j \in \{1, 2, \ldots, n\}\setminus \{k_0\}.$ Since $|\mathcal{Z}_x^c| > 2$ and $|x_j|$ are equal for all $j \in \mathcal{Z}_x^c,$ it follows that $\sum_{j=1}^n sgn(y_j)x_j \neq 0.$ This gives us $y \not\perp_\rho x,$ which contradicts the fact that $x$ is $\rho$-left symmetric. This completes the proof of the necessary part. 
 	
 	
 \end{proof}
 
 Next we characterize the $\rho$-right symmetric points in $\ell_1^n.$ 
 
 \begin{theorem}\label{right; sum}
 	Let $x=(x_1, x_2, \ldots, x_n) \in S_{\ell_1^n}.$ Then $x$ is $\rho$-right symmetric if and only if either of the following conditions hold true: 
 	\begin{itemize}
 		\item[(i)] $x \in Ext(B_{\ell_1^n}).$
 		\item[(ii)] For any two nonempty disjoint sets $A, B \subset \mathcal{Z}_x^c,$ $|\sum_{j\in A} x_j| \neq |\sum_{j\in B} x_j|.$ 
 	\end{itemize}
 \end{theorem}

  \begin{proof}
     We first prove the sufficient part. Suppose that (i) holds true, i.e., $x \in Ext(B_{\ell_1^n}).$ Then $x_i = \pm 1,$ for some $1 \leq i \leq n$ and $x_j=0,$ for all $j \in \{1, 2, \ldots, n\}\setminus \{i\}.$ Suppose that $y \perp_\rho x,$ where $y=(y_1, y_2, \ldots, y_n) \in S_{\ell_1^n}.$ Then from Proposition \ref{ortho; characterization} it is easy to see that $y_i = 0.$ Since $x_j = 0,$ for all $j \neq i,$ it follows that $\sum_{j=1}^n sgn(x_j)y_j =0.$ Thus again from Proposition \ref{ortho; characterization} we obtain that $x \perp_\rho y.$ This proves that $x$ is $\rho$-right symmetric. Now suppose that (ii) holds true. We claim that if $y \perp_\rho x,$ then we have $y_i = 0,$ for all $i \in \mathcal{Z}_x^c.$ If possible, let $y_k \neq 0,$ for some $k \in \mathcal{Z}_x^c.$ Let us consider the two sets $A, B$ as:
     \[A_1 = \{j \in \mathcal{Z}_x^c: sgn(y_j) = +1\}, A_2=\{j \in \mathcal{Z}_x^c: sgn(y_j) = -1\}.\]
     Since $y_k \neq 0,$ for some $k \in \mathcal{Z}_x^c,$ it follows that $A_1 \cup A_2 \neq \emptyset.$ Since $y \perp_\rho x,$ then from Proposition \ref{ortho; characterization}, $\sum_{j =1}^n sgn(y_j)x_j =0.$  Note that whenever $|A_1 \cup A_2| = 1, $ we have $x_k = 0,$ where $k \in \mathcal{Z}_x^c.$ On the other hand, suppose that $|A_1 \cup A_2| \geq 2.$ Then Clearly, we obtain two sets $A, B \subset \mathcal{Z}_x^c$ such that  $|\sum_{j\in A} x_j| =|\sum_{j\in B} x_j|.$ Both these cases are not  possible according to our assumption. So our claim is established. Since $y_j=0,$ for all $j \in \mathcal{Z}_x^c,$ it is easy to see from Proposition \ref{ortho; characterization} that $x \perp_\rho y,$ i.e., $x$ is $\rho$-right symmetric.\\
     
     Now we prove the necessary part. Since $x \in S_{\ell_1^n},$ we have  $\mathcal{Z}_x^c \neq \emptyset.$  If $|\mathcal{Z}_x^c| =1 $ then we have $x \in Ext(B_{\ell_1^n}),$ i.e., (i) holds true. Now let $|\mathcal{Z}_x^c| \geq 2.$  Suppose on the contrary that there exist two nonempty disjoint subsets $A$ and $B$ of $\mathcal{Z}_x^c$ such that $|\sum_{j\in A} x_j| = |\sum_{j\in B} x_j|.$ Without loss of generality assume that $\sum_{j\in A} x_j = \sum_{j\in B} x_j.$ Then  choose $y=(y_1, y_2, \ldots, y_n) \in \ell_1^n$ such that
     \begin{eqnarray*}
     	  y_j &=& 10^j, \, \, j\in A\\
     	       &=& -\frac{1}{10^j}, \, \, j \in B\\
     	       &=& 0, \, \, \quad \quad j \in \{1, 2, \ldots, n\}\setminus (A\cup B).
     \end{eqnarray*} 
     Note that $\sum_{j=1}^n sgn(y_j)x_j = \sum_{j \in A\cup B} sgn(y_j)x_j =0.$ Therefore, from Proposition \ref{ortho; characterization} we obtain that $y \perp_\rho x.$ But one can observe from the construction of $y$ that $\sum_{j=1}^n sgn(x_j)y_j = \sum_{A \cup B} sgn(x_j)y_j \neq 0. $ This shows that $x \not\perp_\rho y.$ Thus we arrive at a contradiction to the fact that $x$ is $\rho$-right symmetric. This completes the proof  the theorem.
  \end{proof}
 
 Combining Theorem \ref{left; sum} and Theorem \ref{right; sum} we note the following:
 
 \begin{theorem}
 	Let $x \in \ell_1^n.$ Then $x$ is $\rho$-symmetric if and only if $x$ is an extreme point of $\ell_1^n.$
 \end{theorem}
 
 In the following example we  describe some $\rho$-left and $\rho$-right symmetric points of $\ell_1^n$ other than the extreme points.
 
 \begin{example}
 	 Suppose that $\mathbb{X}=\ell_1^4$ and let us consider three points of $x_1, x_2, x_3 \in \mathbb{X}$ such that  $x_1=(\frac{1}{2}, 0, 0, -\frac{1}{2}), x_2=(\frac{1}{2}, \frac{1}{3}, 0, -\frac{1}{4})$ and $x_3=(\frac{1}{4}, \frac{1}{4}, \frac{1}{4}, \frac{1}{4}).$  From Theorem \ref{left; sum} it is easy to see that $x_1$ is $\rho$-left symmetric whereas, applying Theorem \ref{right; sum} we have $x_2$ is $\rho$-right symmetric.  On the other hand, $x_3$ is neither $\rho$-left nor $\rho$-right symmetric in $\ell_1^4.$
 \end{example}
 
 In the following two theorems we characterize the $\rho$-left and $\rho$-right symmetric points in $\ell_\infty^n,$ respectively. 
 
 \begin{theorem}\label{left; sup}
 	Let $x=(x_1, x_2, \ldots, x_n)\in S_{\ell_\infty^n}.$  Then $x$ is $\rho$-left symmetric if and only if $x_j =0,$ for all $j \not\in \mathcal{I}_x.$ 
 \end{theorem}
 
 \begin{proof}
 	We first prove the sufficient part. Since $\|x\|=1,$ we have $\mathcal{I}_x\neq \emptyset.$  Let $|\mathcal{I}_x|=1.$ Then $x_i=\pm 1,$ for some $i\in \{1, 2, \ldots, n\}.$   If $x \perp_\rho y,$ for some $y\in S_{\ell_\infty^n}, $ then from Proposition \ref{ortho; characterization}, we get $y_i=0.$ Therefore, $\mathcal{I}_y \subset \{1, 2, \ldots, n\}\setminus \{i\}.$  Since $x_j=0,$ for all $j\in \{1, 2, \ldots, n\} \setminus \{i\},$ it follows from Proposition \ref{ortho; characterization} that $y\perp_\rho x.$ Therefore,  $x$ is $\rho$-left symmetric. Suppose that $|\mathcal{I}_x|\geq 2.$ From Proposition \ref{ortho; characterization} $x \perp_\rho y$ implies that 
 	\begin{equation}
 	\max_{j \in \mathcal{I}_x}\{sgn(x_j)y_j\}+\min_{j \in \mathcal{I}_x}\{sgn(x_j)y_j\}=0.
 	\end{equation}
 	 Suppose that $\max_{j \in \mathcal{I}_x}\{sgn(x_j)y_j\} = sgn(x_k)y_k$ and $\min_{j \in \mathcal{I}_x}\{sgn(x_j)y_j\}=sgn(x_l)y_l,$ for some $k, l\in \mathcal{I}_x.$ Then from Equation (3), it is clear that $|y_k|=|y_l|.$ Now either of the following holds:
 	\begin{itemize}
 		\item[(a)] $|y_k|=|y_l|=1,$ for some $k, l\in \mathcal{I}_x.$
 		\item[(b)] $|y_k|=|y_l|<1,$ for all $k, l\in \mathcal{I}_x.$
 	\end{itemize}
 	If (a) holds, then $k, l \in \mathcal{I}_y$ and consequently, $sgn(y_k)x_k + sgn(y_l)x_l= \max_{j \in \mathcal{I}_y}\{sgn(y_j)x_j\}+\min_{j \in \mathcal{I}_y}\{sgn(y_j)x_j\}=0.$ Thus by Proposition \ref{ortho; characterization}, we get $y \perp_\rho x.$ This implies $x$ is $\rho$-left symmetric. If (b) holds, then $\mathcal{I}_y \cap \mathcal{I}_x =\emptyset.$ From our hypothesis observe that $sgn(y_i)x_i=0,$ for all $i \in \mathcal{I}_y.$ Therefore, $y \perp_\rho x.$ This also shows that $x$ is $\rho$-left symmetric.  \\

 	
 	To show the necessary part suppose on the contrary that there exists $j \in \{1, 2, \ldots, n\}$ such that $0 < |x_j| < 1.$ Then we take $y=(y_1, y_2, \ldots, y_n)$ such that $y_j =1$ and $y_i =0,$ for all $i \in \{1, 2, \ldots, n\} \setminus \{j\}.$ Note that $\mathcal{I}_x \cap \mathcal{I}_y = \emptyset.$ Therefore, $\max_{i \in \mathcal{I}_x}\{sgn(x_i)y_i\}=0=\min_{i \in \mathcal{I}_x}\{sgn(x_i)y_i\}.$ Using Proposition \ref{ortho; characterization},  we have $x \perp_\rho y.$ On the other hand, observe that $\mathcal{I}_y=\{j\}$ and therefore,  $\max_{i \in \mathcal{I}_y}\{sgn(y_i)x_i\}=  \min_{i \in \mathcal{I}_y}\{sgn(y_i)x_i\}=x_j \neq 0.$ Thus we get $y \not\perp_\rho x,$ which contradicts that $x$ is $\rho$-left symmetric.  This completes the proof of the theorem.
 	

 \end{proof}

 \begin{theorem}\label{right; sup}
 	Let $x=(x_1, x_2, \ldots, x_n)\in S_{\ell_\infty^n}.$ Then $x$ is $\rho$-right symmetric if and only if either of the following holds true:
 	\begin{itemize}
 		\item[(i)] $x \in Ext(B_{\ell_\infty^n})$
 		\item[(ii)] for each $j \in \{1, 2, \ldots, n\} \setminus \mathcal{I}_x,$  $0 < |x_j| <1.$ Moreover,  $|x_j| \neq |x_k|,$ for all $j, k \in \{1, 2, \ldots, n\} \setminus \mathcal{I}_x.$ 
 	\end{itemize}
 \end{theorem}

\begin{proof}
	To prove the sufficient part  first assume that (i) holds true. Since $\rho$-orthogonality is preserved under the signed permutation map \cite{W}, we may without loss of generality assume that $x=(1, 1, \ldots, 1).$ Suppose that $y \perp_\rho x,$ for some $y=(y_1, y_2, \ldots, y_n)\in S_{\ell_\infty^n}.$ From Proposition \ref{ortho; characterization}, we observe that  there exist $i, j \in \{1, 2, \ldots, n\}$ such that $y_i=1$ and $y_j=-1.$ Therefore, $\max_{i \in \mathcal{I}_x}\{sgn(x_i)y_i\}=1$ and $\min_{i \in \mathcal{I}_x}\{sgn(x_i)y_i\}=-1.$ From Proposition \ref{ortho; characterization}, we get that $x \perp_\rho y.$ Thus $x$ is $\rho$-right symmetric. Now suppose that (ii) holds true and $y \perp_\rho x,$ for some $y \in S_{\ell_\infty^n}.$ Clearly, $|\mathcal{I}_x| \leq n-1.$ If $|\mathcal{I}_x|=1$ then using Proposition \ref{ortho; characterization} one can see that there does not exists any nonzero $y \in \ell_\infty^n$ such that $y \perp_\rho x.$ Thus  $x$ is $\rho$-right symmetric, vacuously. Let $|\mathcal{I}_x| \geq 2.$ As $y \perp_\rho x,$ from Proposition \ref{ortho; characterization}, we get $\max_{i \in \mathcal{I}_y}\{sgn(y_i)x_i\}+\min_{i \in \mathcal{I}_y}\{sgn(y_i)x_i\} =0.$ This implies that $|x_j|=|x_k|,$ for some $j, k \in \mathcal{I}_y.$ Therefore, From hypothesis we note that $\mathcal{I}_x \cap \mathcal{I}_y \neq \emptyset.$  This implies that there exist $j, k \in \mathcal{I}_x \cap\mathcal{I}_y $ such that $sgn(x_j)y_j=1$ and $sgn(x_k)y_k=-1.$ This shows from Proposition \ref{ortho; characterization} that $x \perp_\rho y.$ Therefore, $x$ is $\rho$-right symmetric. 
	
	
	 To show the necessary part, first suppose on the contrary that $x_j=0,$ for some $i \in \{1, 2, \ldots, n\}.$ Then we choose $y = (y_1, y_2, \ldots, y_n)$ such that $y_j=1$ and $y_k=\frac{1}{10^k},$ for all $k \in \{1, 2, \ldots, n\}\setminus \{i\}.$ One can clearly observe that $y \perp_\rho x,$ whereas $x \not\perp_\rho y.$ This contradicts that $x$ is $\rho$-right symmetric.  Now again we assume on a contrary that $0<|x_j| = |x_k| < 1,$ for some $ j, k \in \{1, 2, \ldots, n\}.$ Then we take $y \in S_{\ell_\infty^n}$ such that $y_j = sgn(x_j)$ and $y_k=-sgn(x_k)$ and $y_i=\frac{1}{10^i},$ for all $i \in \{1, 2, \ldots, n\}\setminus \{j, k\}.$ Then applying Proposition \ref{ortho; characterization}, we have $y \perp_\rho x$ but $x \not\perp_\rho y.$ This contradiction completes the proof of  the necessary part. 
\end{proof}
 
 Combining Theorem \ref{left; sup} and Theorem \ref{right; sup} we note that the extreme points are the only $\rho$-symmetric points on the unit sphere of $\ell_\infty^n.$
 \begin{theorem}
 	Let $x \in S_{\ell_\infty^n}.$ Then $x$ is $\rho$-symmetric point if and only if $x$ is an extreme point of $B_{\ell_\infty^n}.$
 \end{theorem}

We end this article with examples of  $\rho$-left and $\rho$-right symmetric points in $\ell_\infty^n,$ which are not extreme points.
 
 \begin{example}
 	Let us consider  $x_1=(1, 1, 0, 0, -1)$; $x_2 = (1, \frac{1}{2}, \frac{1}{5}, -1, \frac{2}{3})$ and $x_3 = (1, -\frac{1}{3}, 1, \frac{1}{3}, \frac{1}{7})$ are three points in $\ell_\infty^5.$ From Theorem \ref{left; sup} we observe that $x_1$ is a $\rho$-left symmetric point and from Theorem \ref{right; sup} we get that $x_2$ is $\rho$-right symmetric points. On the other hand, it is easy to see that $x_3$ is neither $\rho$-left symmetric nor $\rho$-right symmetric.  
 \end{example}
 
\subsection*{Declarations}
Data sharing is not applicable to this article as no datasets were generated or analysed during the current study.
Authors also declare that there is no financial or non-financial interests that are directly or indirectly related to the work submitted for publication.  On behalf of all authors, the corresponding author states that there is no conflict of interest.

\end{document}